\documentclass[preprint,12pt,3p]{elsarticle}

\usepackage{amssymb}
\usepackage{amsthm}
\usepackage{amsmath}
\usepackage{float}
\usepackage[hidelinks,colorlinks=false]{hyperref}
\newtheorem{theorem}{Theorem}

\newtheorem{corollary}{Corollary}

\newtheorem{definition}{Definition}

\begin{document}

\begin{frontmatter}

\title{Chromatic number, Clique number, and Lov\'{a}sz's bound:\\
In a comparison}

\author[label1, label2]{Hamid Reza Daneshpajouh}
\address[label1]{School of Mathematics, Institute for Research in Fundamental Sciences (IPM),
Tehran, Iran, P.O. Box 19395-5746}
\address[label2]{Moscow Institute of Physics and Technology, Institutsky lane 9, Dolgoprudny, Moscow region, 141700}


\ead{hr.daneshpajouh@ipm.ir, hr.daneshpajouh@phystech.edu}



\begin{abstract}
In the way of proving Kneser's conjecture, L\'{a}szl\'{o} Lov\'{a}sz settled out a new lower bound for the chromatic number. He showed that if neighborhood complex $\mathcal{N}(G)$ of a graph $G$ is topologically $k$-connected, then its chromatic number is at least $k+3$. Then he completed his proof by showing that this bound is tight for the Kneser graph. However, there are some graphs where this bound is not useful at all, even comparing by the obvious bound; the clique number. For instance, if a graph contains no complete bipartite graph $\mathcal{K}_{l, m}$, then its Lov\'{a}sz's bound is at most $l+m-1$. But, it can have an arbitrarily large chromatic number.

In this note, we present new graphs showing that the
gaps between the chromatic number, the clique number, and the Lov\'{a}sz bound can be arbitrarily large. More precisely, for given positive integers $l, m$ and $2\leq p\leq q$, we construct a connected graph which contains a copy of $\mathcal{K}_{l,m}$, and its chromatic number, clique number, and Lov\'{a}sz's bound are $q$, $p$, and $3$, respectively.

\end{abstract}

\begin{keyword}
Chromatic number\sep  topologically $k$-connected \sep neighborhood complex  
\end{keyword}

\end{frontmatter}

\section{Introduction}
Ingenious proof of the well-known Kneser conjecture by Lov\'{a}sz~\cite{Lovasz}, is based on introducing a topological lower bound for the chromatic number of a graph.
To state his bound, let us first recall some definitions and set up our notation. The neighborhood complex $\mathcal{N}(G)$ of a given graph $G$ is the simplicial complex whose simplicies are all subsets of vertices which have a common neighbor. Here and subsequently, the geometric realization of a simplicial complex $\mathcal{K}$ will be denoted by $||\mathcal{K}||$. A topological space $X$ is called $k$-connected if for every $-1\leq m\leq k$, each continuous mapping of the $m$-dimensional sphere $\mathbb{S}^m$ into $X$ can be extended to a continuous mapping of the $(m + 1)$-dimensional ball $\mathbb{B}^{m+1}$. Here $\mathbb{S}^{-1}$ is interpreted as $\emptyset$ and $\mathbb{B}^{-1}$ a single point, and so $(-1)$-connected means nonempty. In another words, a non-empty topological space $X$ is $k$-connected if all the homotopy group $\pi_i(X)$ are trivial, for all $i= 0, \cdots, k$. The largest $k$ such that $X$ is $k$-connected is called the connectivity of $X$, denoted by $conn(X)$. A simplicial complex is called $k$-connected if
its geometric realization is $k$-connected. Finally, the chromatic number $\chi (G)$ of a graph $G$ is the smallest number of colors needed to color the vertices of $G$ so that no two adjacent vertices share the same color. Now, we are in a position to recall the Lov\'{a}sz bound.


\begin{theorem}[ Lov\'{a}sz's bound~\cite{Lovasz}]
If $\mathcal{N}(G)$ is $k$-connected, then $\chi(G)\geq k+3$.
\end{theorem}
On the other hand, recall that a known and obvious lower bound of any graph $G$ is the clique number $\omega (G)$ of $G$, i.e. $\chi(G)\geq\omega(G)$. Now the natural questions arise: In general, how much the Lov\'{a}sz bound can be close to the chromatic number? Or in compassion by the clique number, can the Lov\'{a}sz  bound give a better approximation than the clique number, in general? There are Several examples that shows the Lov\'{a}sz bound is not very useful. For instance, if a graph contains no complete bipartite graph $\mathcal{K}_{l, m}$, then its Lov\'{a}sz's bound is at most $l+m-1$~\cite{Csorba1}. But, it can have an arbitrarily large chromatic number~\cite{Erdos}.

Even, in comparison with the clique number, we cannot hope for a better result.
For example, in~\cite{Kahle} author shows that the connectivity of the neighborhood complex of a random graph is almost always between $\frac{1}{2}$ and $\frac{2}{3}$ of the expected clique number. As another example, in~\cite{Csorba} authors construct a sequence of graphs ${\{G(m)\}}_{m=1}^{\infty}$ such that:
$$\lim_{m\to\infty}conn (||\mathcal{N}(G(m))||) = c\,\, ,\,\text{but}\, \lim_{m\to\infty}\omega (G(m)) = \infty,$$
where $c$ is a constant number.

\section{Main part}
In this section, we state and discuss the main results of this note.
\begin{definition}
Let $H$ and $K$ be connected non-bipartite graphs with disjoint vertex sets $V(H)$ and $V(K)$, respectively. Fix $x\in V(H)$, and $y\in V(K)$. Let $G_{H_x, K_y}$ be obtained by taking $H$, $K$ and attaching a path of length $2$ by one its endpoints to $x$ and the other one to $y$. In other words, $G_{H_x, K_y}$ is a graph with vertex set $V(H)\cup V(K)\cup\{z\}$ and edge set
$E(H)\cup E(K)\cup\{\{x,z\}, \{y, z\}\}$ where $z\notin\left(V(H)\cup V(K)\right)$.
\end{definition}
Before we proceed, we need to recall a definition.
Let ${\{(X_i,x_i)\}}_{i\in I}$ be a family of pointed topological spaces, i.e., topological spaces with distinguished base points $x_i$. The wedge sum $\bigvee _{i}X_{i}$ of the family is the quotient space of the disjoint union of members of the family by the identification all $x_i$, i.e, $\bigvee _{i}X_{i}=\coprod _{i}X_{i}\;/{\{x_i: i\in I\}}$.  In other words, the wedge sum is the joining of several spaces at a single point. Note that, in general, the wedge sum depends on the choice of base points. However, for simplicity, we omit base points from the notation when they are clear from context. Now we are in a position to state and prove our main theorem.
\begin{theorem}
Let $H$ and $K$ be connected non-bipartite graphs with disjoint vertex sets $V(H)$ and $V(K)$, respectively. Fix $x\in V(H)$, and $y\in V(K)$. Then $||\mathcal{N}(G_{H_x,K_y})||$ is homotopy equivalent to $||\mathcal{N}(H)||\vee||\mathcal{N}(K)||\vee\mathbb{S}^1$. In particular, 
$$conn (||\mathcal{N}(G_{H_x,K_y})||)=0.$$
\end{theorem}

\begin{proof}
Throughout the proof, for a graph $F$ and a vertex $\nu\in V(F)$, $CN_F(\nu)$ denotes for the set of all neighbors of $\nu$ in $F$.
Let us write $\mathcal{K}= \mathcal{N}(G_{H_x,K_y})$. We can split $\mathcal{K}$ into three simplicial subcomplexes as follows.
\begin{equation*} 
\begin{split}
\mathcal{K} = & \underbrace{\left(\mathcal{N}(H)\cup\{A\cup\{z\} : A\subseteq CN_{H}(x)\}\right)}_{\mathcal{K}_1}\bigcup \\
 & \underbrace{\left(\mathcal{N}(K)\cup\{A\cup\{z\} : A\subseteq CN_{K}(y)\}\right)}_{\mathcal{K}_2}\bigcup\\
 & \underbrace{\{\emptyset, \{x\}, \{y\}, \{x,y\}\}}_{\mathcal{K}_3}.
\end{split}
\end{equation*}
Here is an illustration of the structures of $||\mathcal{K}_1||$, $||\mathcal{K}_2||$, and $||\mathcal{K}_3||$:
\begin{figure}[H]
    \centering
    \includegraphics[scale=0.37]{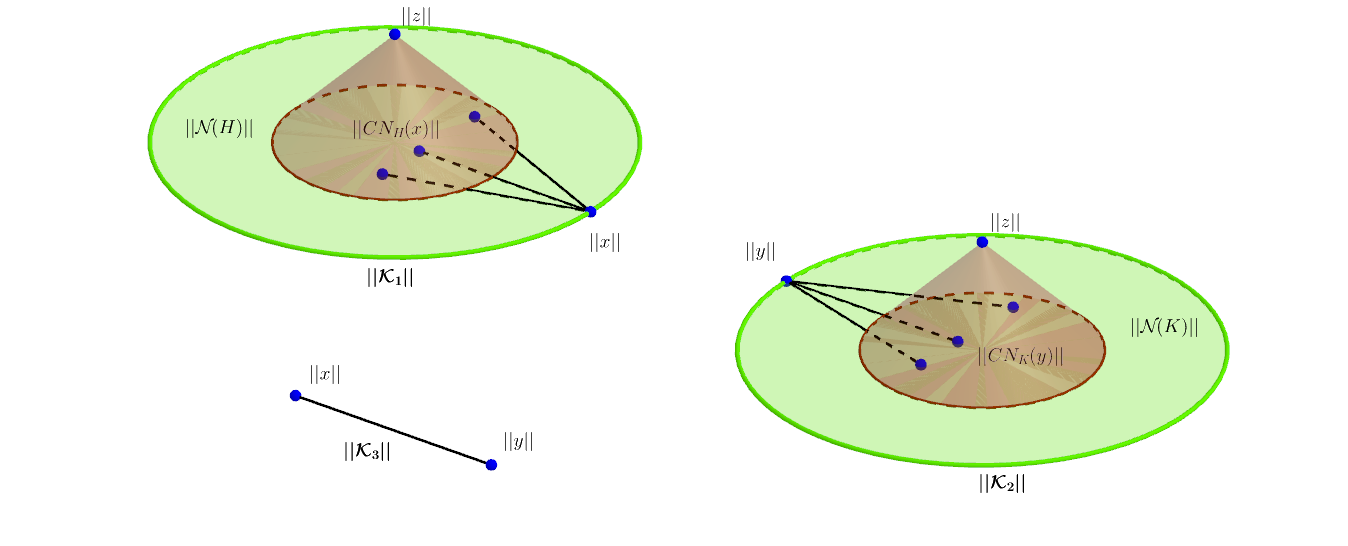}
\end{figure}
Note that $\mathcal{N}(H)\subset\mathcal{K}_1$, and $\mathcal{N}(K)\subset\mathcal{K}_2$ are created by attaching cones over $CN_{H}(x)$ and $CN_{K}(y)$, respectively. Also, $\mathcal{K}_1\cap\mathcal{K}_2 =\{z\}$. So, we can contract $||\mathcal{K}||$ to $X=||\mathcal{N}(H)||\vee||\mathcal{N}(K)||\cup||\mathcal{K}_3||$. On the other hand, it is not hard to see that the neighborhood complex of a connected non-bipartite graph is always path connected. So, $||\mathcal{N}(H)||$ and $||\mathcal{N}(K)||$ are path connected as well. Thus, $X$ is path-connected. Finally, we can move gradually the endpoints of $\mathcal{K}_3$, $||x||$ and $||y||$, inside $X$ to the point $||z||$. In summary, we have
$$||N(G_{H_x, K_y})||\simeq X\simeq||\mathcal{N}(H)||\vee||\mathcal{N}(K)||\vee\mathbb{S}^1$$
Here is an illustration of the procedure that we described above:
\begin{figure}[H]
    \centering
    \includegraphics[scale=0.36]{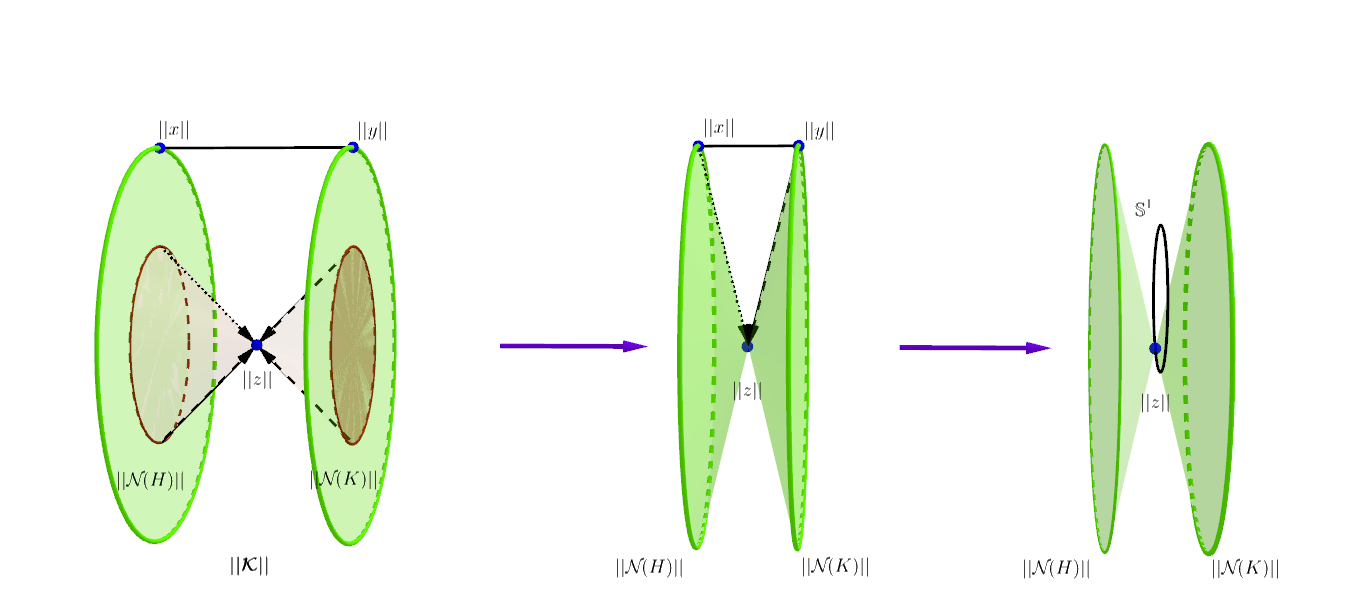}
\end{figure}
For the second part, we are going to use van Kampen's theorem~\cite{Hatcher}. Since any simplicial complex is locally contractible, by using van Kampen's theorem, we get
\begin{equation*} 
\begin{split}
\pi_1 (||\mathcal{N}(G_{H_x,K_y})||) & = \pi_1 (||\mathcal{N}(H)||\vee||\mathcal{N}(K)||\vee\mathbb{S}^1)\\
& =  \pi_1 (||\mathcal{N}(H)||)\ast\pi_1(||\mathcal{N}(K)||)\ast\pi_1(\mathbb{S}^1)\\
& =  \pi_1 (||\mathcal{N}(H)||)\ast\pi_1(||\mathcal{N}(K)||)\ast\mathbb{Z}\neq 0.
\end{split}
\end{equation*}
On the other hand $||\mathcal{N}(G_{H_x,K_y})||$ is pathconnected, i.e., $\pi_0 (||\mathcal{N}(G_{H_x,K_y})||)=0$. Thus, $conn (||\mathcal{N}(G_{H_x,K_y})||)=0.$
\end{proof}
As a corollary of above theorem, we give a new example of a connected graph containing an arbitrary large bipartite subgraph with additional properties that
gaps between its chromatic number, clique number, and Lov\'{a}sz bound are arbitrarily large. Consider positive integers $l, m$ and $2\leq p\leq q$. Let $T$ be a triangle-free graph with the chromatic number $q$; for construction such graphs see~\cite{Daneshpajouh, Erdos, Mycielski}. 
Let $\mathcal{K}_p$ and $\mathcal{K}_{l, m}$ be the complete graph and complete bipartite graph, respectively. Also, without loss of generality, suppose that $V(\mathcal{K}_p)\cap V(\mathcal{K}_{l,m})\cap V(T)=\emptyset$. Fix $a, b\in V(\mathcal{K}_m)$, $c\in V(\mathcal{K}_{l,m})$, and $d\in V(T)$. Now, let $H$ be obtained by taking $\mathcal{K}_{l,m}$, $\mathcal{K}_p$, and $T$ and connecting $a$ to $c$, and $b$ to $d$, by two different edges. In other words, $H$ is a graph with the following vertex set and edge set:
\begin{equation*} 
\begin{split}
& V(H) = V(\mathcal{K}_m)\cup V(\mathcal{K}_{l,m})\cup V(T),\\
& E(H) = E(\mathcal{K}_m)\cup E(\mathcal{K}_{l,m})\cup E(T)\cup\{a,c\}\cup\{b, d\}.
\end{split}
\end{equation*}
Now, fix an arbitrary vertex $s\in V(H)$. It is easy to check that the graph $G_{H_s,H_s}$ has the following properties, $\chi (G_{H_s, H_s}) = q$, $\omega (G_{H_s, H_s}) = p$, $\mathcal{K}_{l, m}\subseteq G_{H_s, H_s}$. But, by Theorem $2$, its Lov\'{a}sz bound is trivial, i.e., $3$.
In conclusion, we have the following corollary.
\begin{corollary}
For given positive integers $l, m$ and $2\leq p\leq q$, there is a graph which contains a copy of $\mathcal{K}_{l,m}$, and its chromatic number, clique number, and Lov\'{a}sz's bound are $q$, $p$, and $3$, respectively.
\end{corollary}

\section*{Acknowledgements}
This is a part of the author's Ph.D. thesis, under the supervision of Professor Hossein Hajiabolhassan. I would like to express my sincere gratitude to him, for his constant support, guidance and motivation.

\end{document}